\documentclass[12pt]{amsart}

\usepackage{amssymb,amsmath,amsbsy, amsthm}
\usepackage[english]{babel}
\usepackage{amscd}
\usepackage{pb-diagram}
\usepackage{pstricks,pst-node}
\usepackage{amsfonts}
\usepackage{latexsym}
\usepackage{pstricks,pst-node}
\usepackage{tikz}
\usepackage{fullpage}

\usepackage{hyperref}

\newtheorem{theorem}{Theorem}[section]
\newtheorem{lemma}[theorem]{Lemma}
\newtheorem{corollary}[theorem]{Corollary}
\newtheorem{claim}[theorem]{Claim}

\newtheorem{proposition}[theorem]{Proposition}
\newtheorem{example}[theorem]{Example}

\theoremstyle{definition}

\numberwithin{equation}{section}


\newcommand{\N}{\mathbb{N}}

\newcommand{\BB}{\widetilde{B}}

\newcommand{\ideal}{\mathcal{I}}

\def\cc{{\mathcal C}}
\def\ideal{{\mathcal I}}
\def\baire{{\N}^{\N}}
\def\fin{\mbox{\sf Fin}}
\def\cantor{2^{\N}}

\def\pp{$p^+$}
\def\qq{$q^+\,$}
\def\su{\subseteq}

\begin{document}
\date{}
\title{Some topological and combinatorial properties  preserved by  inverse limits}         
\author{Javier Camargo}
\address{Javier Camargo\\
\newline
Escuela de Matem\'aticas, Facultad de Ciencias, Universidad Industrial de
Santander, Ciudad Universitaria, Carrera 27 Calle 9, Bucaramanga,
Santander, A.A. 678, COLOMBIA.}
\email{ jcamargo@saber.uis.edu.co}

\author{Carlos Uzc\'ategui}
\address{Carlos Uzc\'ategui\\
\newline 
Escuela de Matem\'aticas, Facultad de Ciencias, Universidad Industrial de
Santander, Ciudad Universitaria, Carrera 27 Calle 9, Bucaramanga,
Santander, A.A. 678, COLOMBIA\\
\newline
\newline
Centro Interdisciplinario de L\'ogica y \'Algebra, Facultad de Ciencias, Universidad de Los Andes, M\'erida, VENEZUELA.}
\email{ cuzcatea@saber.uis.edu.co}



\begin{abstract}
We show that  the following properties are preserved under inverse limits: countable fan-tightness, \qq, discrete generation and selective separability. We also present several examples based on inverse limits of countable spaces. 
\end{abstract}

\maketitle

\section{Introduction}
A space $X$ is Fr\'echet if for all $A\su X$ and $x\in \overline A$  there is a sequence in $A$ converging to $x$. It is well known that this property is not productive. A stronger form of Fr\'echetness, which is productive, is bisequentiality. In  the class of countable spaces,  bisequentiality is closed related to two properties of combinatorial nature.
A space $X$ has {\em countable fan-tightness} \cite{ArhanBella1996}, if for all $x$ and all sequence of subsets $A_n\su X$ with $x\in \overline{A_n}$, there is $K_n\su A_n$ finite, for each $n\in\N$, such that $x\in \overline{\bigcup_n K_n}$. $X$ is a  \qq space,  if for  all $A\su X$, all $x\in \overline{A}$ and any partition $(F_n)_n$ of $A$ into finite sets, there is $S$ such that $S\su A$, $x\in \overline{S}$ and $S\cap F_n$  has at most one element for each $n$.  Every countable sequential space if \qq\ (see \cite[proposition 3.3]{Todoruzca2000}). Moreover, a countable regular space  with analytic topology (see the definition in the next section) is bisequential iff it has countable fan-tightness   and is \qq (see \cite[Theorem 7.65]{Todor2010}).  

Under CH there is a pair of Fr\'echet spaces with countable fan-tightness whose product  is Fr\'echet but does not have countable fan-tightness \cite{Simon98}.  However, there are not such examples with analytic topology (\cite{Todoruzca2000}). We do not know whether \qq is productive, nevertheless, we will show that the product of a countable metric space and a \qq space is again \qq. 

It is known that the inverse limit of Fr\'echet spaces might not be even sequential (see for instance Example \ref{ejem-nosequential}).  Nogura \cite{Nogura85,Nogura85b} studied when the inverse limit of Fr\'echet spaces is also Fr\'echet. We continue this line of investigation by studying some  other properties of  countable spaces which are preserved under inverse limits. We focus on countable fan-tightness,  \qq, discrete generation and selective separability.   We recall that $X$ is {\em discretely generated} \cite{DTTW2002}, denoted DG,  if 	for every $A\su X$ and $x\in\overline{A}$, there is $E\su A$ discrete such that $x\in \overline{E}$. $X$ is {\em selectively separable} \cite{Scheeper99}, denoted $SS$,  if for any sequence $(D_n)_n$ of dense subsets of $X$ there is $K_n\subseteq D_n$ finite such that $\bigcup_n K_n$ is dense in $X$.   
DG and SS are not productive (see \cite{murtinova2006} and \cite{BarmanDow2011} respectively).  
Countable fan-tightness implies SS \cite{Bella_et_al2008} and also DG \cite{BellaSimon2004} but it does not imply property \qq (see Example \ref{ejemplo2}).  

We recall that an heredirary class of spaces is closed under countable products iff it is closed under finite products and inverse limits.    Countable fan-tightness, DG and \qq are hereditary properties, but SS is not.  We will show that all those four properties are preserved under inverse limits.   The proof for  \qq is presented for countable spaces with analytic topology.  Our original motivation was  to construct countable crowded spaces with analytic topology satisfying some of those  properties.  For instance, we construct a countable crowded space with analytic topology which is DG, SS and   \qq, but does not have countable fan-tightness (see Example \ref{ejem-lastres}).


\section{Preliminaries}
An {\em ideal} on a set $X$ is a nonempty collection $\ideal$ of subsets of
$X$ satisfying: (i) $A\su B$ and $B\in \ideal$, then $A\in \ideal$ and  (ii) If $A, B\in \ideal$, then $A\cup B \in \ideal$.
The ideal is not trivial if $X\not\in \ideal$ and when every finite subset of $X$ belongs to $\ideal$ the ideal is called free. 
We will assume that all ideals are free.  
If $A\su X$, then $\ideal\restriction A$ is the ideal on $A$ given by $\{B\su
A:\; B\in \ideal\}$.  Notice that $\ideal\restriction A$ is not trivial 
only when $A\not\in \ideal$.  We denote by $\fin$  the ideal of finite subsets of $\N$.   Let   $X$  be a topological space and $x\in X$ non isolated. The neighborhood ideal of $x$ is defined as follows:
\begin{equation}
\label{Idealx}
\ideal_x=\{A\su  X:\; x\not\in \overline{A\setminus\{x\}}\}.
\end{equation}
A ideal $\ideal$ is \pp, if for every decreasing sequence $(A_n)_n$ of sets not in
$\ideal$, there is $A\not\in \ideal$ such that $A\su^* A_n$ for all
$n\in\N$ (where $A\su^*B$ means that  $A\setminus B$ is finite). It is easy to check that $X$ has countable fan-tightness iff  $\ideal_x$ is \pp\ for all $x\in X$.  
	
The notion of a \qq space can be expressed as a property of the ideals $\ideal_x$. An ideal $\ideal$ over a countable set is \qq, if for every $A\not\in \ideal$ and every partition $(F_n)_n$ of $A$ into finite sets, there is $S\not\in\ideal$
such that $S\su A$ and $S\cap F_n$  has at most one element for each $n$.   Thus $X$ is a \qq space, if $\ideal_x$ is \qq for all $x$.  This notion was motivated by Ramsey theoretic properties of ideals (see \cite{HMTU2013, Todor2010}). 

We use \cite{Kechris94} as a general reference for all descriptive set theoretic notions. A Polish space is a completely metrizable and separable space.  A subset $A\su X$ of a Polish space $X$ is called {\em analytic}, if it is a
continuous image of a Polish space. Equivalently, if there is a
continuous function $f:\baire\rightarrow X$ with range $A$, where
$\baire$ is the space of irrationals. Every Borel subset of a Polish space is analytic. We say
that a topology $\tau$ over a countable set $X$ is {\em analytic},
if $\tau$ is analytic as a subset of the Cantor cube $2^X=\{0,1\}^{X}$
(identifying subsets of $X$ with characteristic functions). A regular countable space has analytic topology iff it can be embedded in $C_p(\baire)$   (see \cite{todoruzca, Todoruzca2000,Todoruzca2014} for more information about analytic topologies). If there is a base
$\mathcal{B}$ of $X$ such that $\mathcal{B}$ is a $F_\sigma$ subset of
$2^X$, then we say that $X$ has a $F_\sigma$ base.
Typical examples of spaces with analytic topology are  the countable sequential fan and   Arens' space.   A space is {\em crowded} if does not have isolated points.

Given a sequence of spaces $X_n$, $n\in \N$ and continuous functions $f^{n+1}_n: X_{n+1}\rightarrow X_n$, the {\em inverse limit} is the following subspace of the product $\prod_n X_n$:
\[
X_\infty=\{(x_n)_n\in \prod_n X_n:\; x_{n}=f^{n+1}_n(x_{n+1})\; \mbox{for all $n\in\N$} \},
\]
usually denoted by $\underleftarrow{\lim}\{X_n;f_n^{n+1}\}_{n\in\N}$. Let $X$ be a space and  $f\colon X\to X$ a continuous map. When  $X_n=X$ and $f_n^{n+1}=f$, for each $n\in \N$, then the inverse limit is denoted by $\underleftarrow{\lim}\{X;f\}$.
The projection functions $\pi_n:X_\infty\rightarrow X_n$ are defined by $\pi_n(x_m)_m=x_n$. We denote $f_{n}^m=f_{m-1}^m\circ f_{m-2}^{m-1}\circ ... \circ f_n^{n+1}$ whenever $m>n$.
The following is a basic fact that will be used several times  (see e.g. \cite[Proposition 2.5.5]{engelking}).

\begin{proposition}
	\label{baseIL}
	Given an inverse sequence $\{X_n ; f^{n+1}_n\}_{n\in\N}$, the family of subsets  $\pi^{-1}_n(V)$, for $V\su X_n$ open and $n\in \N$, is  a base for $X_\infty$.
\end{proposition}

\begin{proposition}
	\label{inverseAnalytic}
	Let  $\{X_n ; f^{n+1}_n\}_{n\in\N}$ be an inverse sequence of countable  spaces with  analytic topology. Then the topology of any countable subset of $X_\infty $  is analytic. Moreover, if each $X_n$ has a $F_\sigma$ base, then every countable subset of $X_\infty$ also has a $F_\sigma$ base. 
\end{proposition}

\proof  Let $\tau_n$ be the topology on $X_n$, for $n\in \N$.   Let $Y\su X_\infty$ be countable. For each $n\in\N$, consider the function $F_n: 2^{X_n}\rightarrow  2^Y$ given by 
\[
F_n(V)=\pi^{-1}_n(V)\cap Y= \{(y_k)_k\in Y: \; y_n \in V\}, \; \;\mbox{for $V\su X_n$}.
\]
Then each $F_n$ is continuous. Since we are assuming that $\tau_n$ is an analytic subset of $2^{X_n}$, then $F_n[\tau_n]$ is also analytic. Thus $\bigcup_n F_n[\tau_n]$ is an analytic  base for the topology of $Y$. Therefore the topology of $Y$ is also analytic (see \cite[Proposition 3.2]{todoruzca}).
Notice that if $\mathcal{B}_n$ is a $F_\sigma$ base for $\tau_n$, then $\bigcup_n F_n[\mathcal{B}_n]$ is $F_\sigma$ base for $X_\infty$. 
\endproof

\section{Property \qq}

As we said in the introduction, we do not know if property \qq is productive. In this section we look at that issue.  We only treat the case when the spaces have analytic topolgy, since we need to use the following theorem.

\begin{theorem}
\label{jalali}
(Jalali-Naini, Talagrand \cite[Theorem 1, pag 32]{Todor97})
Let $\ideal$ be an ideal   over a countable set  $X$ containing all finite subsets of $X$.  Suppose $\ideal$  has the Baire property as a subset of $2^X$. Then there is a partition $(K_n)_n$ of $X$  into finite sets such that $\bigcup_{i\in A}  K_i \not\in\ideal$ for all  infinite $A\subseteq X$.  
\end{theorem}


We use the following consequence of the previous theorem. 

\begin{lemma}
\label{jalaili2}
Let $X$ be a countable regular space with analytic topology. Let $A\subset X$ and $x\in \overline{A}\setminus A$. Suppose $(G_k)_k$ is a partition of $A$ into finite sets. Then there is $B\subseteq \N$ such that 
$$
x\in (\overline{\bigcup_{k\in B} G_k}) \cap (\overline{\bigcup_{k \not\in B}G_k}).
$$
\end{lemma}

\proof Let $A$ and $x$  be as in the hypothesis.  Let $\ideal_x$ be the ideal given by \eqref{Idealx} on $A$ as a subspace of $X$.  Since the topology of $X$ is analytic, then $\ideal_x$ is analytic and thus has the Baire property. Let $(K_i)_i$ be a partition of $A$ into finite sets as given by Theorem \ref{jalali} applied to $\ideal_x$. Recursively define $B$ such that there are infinite many $i$ such that $K_i\subset \bigcup_{k\in B} G_k$ and also infinitely many $i$ such that $K_i\subset \bigcup_{k\not\in B} G_k$. 
\endproof

\begin{theorem}
\label{productoqplus}
Let $X$ be a countable \qq regular space with analytic topology and $Y$ a countable regular space with a countable base. Then $X\times Y$ is \qq. 
\end{theorem}

\begin{proof} Since $Y$ is second countable, then its topology is analytic and so is the topology of $X\times Y$. 
Let $p=(x_0,y_0)\in X\times Y$, $A\subseteq X\times Y$ such that $p\in\overline{A}$, and $\{F_i :i\in\N\}$ be a pairwise disjoint family of finite sets such that $A=\bigcup_{i\in\N}F_i$. By a repeated application of Lemma \ref{jalaili2}, there is a partition  $\{\mathcal{L}_i : i\in\N\}$ of $\N$ such that $\mathcal{L}_i$ is infinite and $p\in \overline{\bigcup_{n\in\mathcal{L}_i}F_n}$ for each $i\in\N$. Fix a countable local base $(V_k)_k$ at $y_0$. We claim that for each $i\in\N$, there exists $S_i\subseteq \bigcup_{n\in \mathcal{L}_i}F_n$ such that:
		
		1. $|S_i\cap F_n|\leq 1$, for each $n\in\mathcal{L}_i$; 
		
		2. $y_0\in\overline{\pi_1(S_i)}$; and
		
		3. $S_i\subseteq X\times V_i$.
	
Let  $W_i=X\times V_i$, then  $p\in \overline{\bigcup_{n\in\mathcal{L}_i}(F_n\cap W_i)}$. Let $G_1=\pi_1(F_1\cap W_i)$ and $G_{n+1}=\pi_1(F_{n+1}\cap W_i)\setminus G_n$, for each $n\in\mathcal{L}_i$. It is clear that $\bigcup_{n\in\mathcal{L}_i}G_n=\bigcup_{n\in\mathcal{L}_i}(F_n\cap W_i)$ and $G_i\cap G_j=\emptyset$ whenever $i\neq j$. Since $X$ is a \qq space, there is $R_i\subseteq \bigcup_{n\in\mathcal{L}_i}G_n$, $|R_i\cap G_n|\leq 1$ for each $n\in\mathcal{L}_i$, and $x_0\in\overline{R_i}$. It is easy to see that there exists $S_i\subseteq \bigcup_{n\in\mathcal{L}_i}(F_n\cap W_i)$ such that $|S_i\cap(F_n\cap W_i)|\leq 1$, for each $n\in\mathcal{L}_i$, and $\pi_1(S_i)=R_i$. The proof of the claim is complete.
	
Finally, let $S=\bigcup_{i\in\N}S_i$. Observe that $S\subseteq \bigcup_{n\in\N}F_n$ and $|S\cap F_n|\leq 1$, for each $n\in\N$. We  have to show  that $p\in \overline{S}$. Let $U=O\times V_k$ be an open subset of $X\times Y$ such that $p\in U$. Since $x_0 \in\overline{R_k}$, $R_k\cap O\neq\emptyset$. Thus, there exists $z\in S_k$ such that $\pi_1(z)\in R_k\cap O$. Since $S_k\subseteq X\times V_k$, then $\pi_2(z)\in V_k$. Therefore, $z\in S\cap U$ and $p\in\overline{S}$.
	
\end{proof}

Now we will show that the inverse limit of \qq spaces with analytic topologies is also \qq. First we show an auxiliary result. 

\begin{lemma}\label{lema000}
Let $\{X_n;f_n^{n+1}\}_{n\in\N}$ be an inverse sequence with inverse limit $X_{\infty}$. Let $A\subseteq X_{\infty}$ and let $p=(p_n)_{n\in\N}\in X_{\infty}$ such that $p\in \overline{A}\setminus A$. If there exists an increasing sequence $(n_k)_{k\in\N}$ in $\N$ and $O_k\subseteq  X_\infty$ open with $p\in O_k$ such that $p_{n_k}\notin \overline{\pi_{n_k}(A\cap O_k)\setminus \{p_{n_k}\}}$, for each $k\in\N$, then there is a sequence $(a_i)_{i\in\N}$ in $ A\setminus \{p\}$ such that $\lim_{i\to\infty}a_i=p$.
\end{lemma}

\begin{proof}
Let $U_{n_k}$ be an open subset of $X_{n_k}$ such that $p_{n_k}\in U_{n_k}$ and $U_{n_k}\cap (\pi_{n_k}(A\cap O_k)\setminus\{p_{n_k}\})=\emptyset$, for each $k\in\N$. Since $p\in \pi_{n_k}^{-1}(U_{n_k})$ and $A\cap O_k\not\in \ideal_p$, we have that $\pi_{n_k}^{-1}(U_{n_k})\cap ((A\cap O_k)\setminus \{p\})\neq\emptyset$, for each $k\in\N$. Let $z_k\in \pi_{n_k}^{-1}(U_{n_k})\cap ((A\cap O_k)\setminus \{p\})$, for each $k\in\N$.  We will show that  $\lim_{k\to\infty}z_k=p$.
Let $W$ be an open subset of $X_{\infty}$ such that $p\in W$. Hence, there exist $i_0\in\N$ and an open $V_{i_0}\subseteq X_{i_0}$ such that $p\in \pi_{i_0}^{-1}(V_{i_0})\subseteq W$ (see Proposition \ref{baseIL}). Let $n_l\geq i_0$. It suffices to show that $z_j\in \pi_{i_0}^{-1}(V_{i_0})$, for all $j\geq l$. Fix $j\geq l$. Since $f_{i_0}^{n_j}\colon X_{n_j}\to X_{i_0}$ is a continuous map, there is an open $V_{n_j}\subseteq X_{n_j}$ such that $p_{n_j}\in V_{n_j}$ and $f_{i_0}^{n_j}(V_{n_j})\subset V_{i_0}$. Thus, $\pi_{n_j}^{-1}(V_{n_j})\subseteq \pi_{i_0}^{-1}(V_{i_0})$.  We claim that 
\begin{equation*}
\pi_{n_j}^{-1}(U_{n_j})\cap ((A\cap O_j)\setminus\{p\})\subseteq \pi_{n_j}^{-1}(V_{n_j}\cap U_{n_j})\cap ((A\cap O_j)\setminus\{p\}).
\end{equation*} 
In fact, let $x\in \pi_{n_j}^{-1}(U_{n_j})\cap ((A\cap O_j)\setminus\{p\})$. Then, $\pi_{n_j}(x)\in U_{n_j}\cap \pi_{n_j}(A\cap O_j)$. Since $U_{n_j}\cap (\pi_{n_j}(A\cap O_j))\setminus \{p_{n_j}\})=\emptyset$, $\pi_{n_j}(x)=p_{n_j}$. Thus, $x\in \pi_{n_j}^{-1}(V_{n_j}\cap U_{n_j})$ and we are done. 

Since $z_j\in \pi_{n_j}^{-1}(U_{n_j})\cap (A\cap O_j)\setminus \{p\})$, then from the claim we have that $z_j\in \pi_{j}^{-1}(V_{n_j})$. Therefore, $z_j\in \pi_{i_0}^{-1}(V_{i_0})$, for each $j\geq l$.
\end{proof}


\begin{theorem}\label{teoremaq}
Let $\{X_n;f_n^{n+1}\}_{n\in\N}$ be an inverse sequence of countable spaces with an analytic topology. If $X_n$ is a \qq space, for each $n\in\N$, then every countable subspace of $X_\infty$ is \qq.
\end{theorem}

\begin{proof}
Let $Y$ be a countable subset of $X_{\infty}$, and $p=(p_n)_{n\in\N}\in Y$. Let $A\subseteq Y$ such that $p\in \overline{A}$, and $\{F_i : i\in\N\}$ be a pairwise disjoint family of finite subsets of $Y$ such that $A=\bigcup_{i\in\N}F_i$. If there exists a sequence $(z_n)_{n\in\N}$ in $A\setminus\{p\}$ such that $\lim_{n\to\infty}z_n=p$, then it is not difficult to show that there is $S\subseteq A$ such that $p\in \overline{S}$ and $|S\cap F_n|\leq 1$, for each $n\in\N$. 
	Hence, we assume that such sequence $(z_n)_n$ does not exist and thus,  by Lemma \ref{lema000}, there exists $i_0\in\N$ such that $p_j\in \overline{\pi_j(A)\setminus\{p_j\}}$, for each $j\geq i_0$. Let $n_1\geq i_0$. We define a sequence $(G_i)_i$ of subsets of $X_{n_1}$  as follows: Let  $G_0=\pi_{n_1}(F_0)$ and, for each $i>0$, let 
$$
G_i=\pi_{n_1}(F_i)\setminus \bigcup_{j=0}^{i-1}G_j.
$$
Notice that
\begin{itemize}
\item $G_n$ is finite, for each $n\in\N$;
\item $\bigcup_{n\in\N}\pi_{n_1}(F_n)=\bigcup_{n\in\N}G_n$;
\item $G_i\cap G_j=\emptyset$, whenever $i\neq j$.
\end{itemize}
We are going to define   an increasing  sequence $(n_i)_i$ of natural numbers, a sequence $(B_i)_i$ of subsets of $\N$ and a sequence $(S_i)_i$  of sets with $S_i\subseteq X_{n_i}$, for all $i$,  such that 
\begin{enumerate}
		\item $(B_i)_i$ is pairwise disjoint.
		\item $\bigcup_{k\in B_i}G_k \not\in \ideal_{p_{n_i}}$  and $\bigcup_{k\not\in B_i}G_k \not\in \ideal_{p_{n_i}}.$
		\item $S_i\not\in \ideal_{p_{n_i}}$, $S_i\subset \bigcup_{k\in B_i}G_k$ and $|S_i\cap G_k|\leq 1$, for all $k\in B_i$.
		\item $\bigcup_{k\not\in B_i}F_k\not\in \ideal_p$.
\end{enumerate}
Let us suppose  that  $(B_i)_i$, $(S_i)_i$ and $(n_i)_i$ have being constructed and we finish the proof. 
Let $S_i^*\subseteq \bigcup_{k\in B_i}F_k$ be such that
	\begin{itemize}
		\item $|S_i^*\cap F_k|\leq 1$, for each $k\in B_i$;
		\item $\pi_{n_i}(S_i^*)=S_i$.
	\end{itemize} 
Let $S=\bigcup_{i\in\N}S_i^*$ and $B=\bigcup_{i\in\N} B_i$.  Then $|S\cap F_n|\leq 1$, for each $n\in B$. We will  be done  if we show  that $p\in \overline{S}$. Let $W$ be an open subset of $X_{\infty}$ such that $p\in W$. By Proposition \ref{baseIL}, there exist $k\in\N$ and an open set $V_k\subseteq X_k$ such that $p_k\in V_k$ and $\pi_k^{-1}(V_k)\subseteq W$. Let $l\in\N$ such that  $n_l>k$. Since $f_k^{n_l}\colon X_{n_l}\to X_k$ is a continuous map, there is an open set $V_{n_l}\subseteq X_{n_l}$ such that $p_{n_l}\in V_{n_l}$ and  $f_k^{n_l}(V_{n_l})\subseteq V_k$. Hence, $\pi_{n_l}^{-1}(V_{n_l})\subseteq \pi_{k}^{-1}(V_k)\subseteq W$. Since $S_l\not\in \ideal_{p_{n_l}}$, $V_{n_l}\cap S_l\neq\emptyset$. Then   $S_l^*\cap\pi_{n_l}^{-1}(V_{n_l})\neq\emptyset$. Thus, $S\cap \pi_{n_l}^{-1}(V_{n_l})\neq\emptyset$ and, since $\pi_{n_l}^{-1}(V_{n_l})\subseteq W$, $S\cap W\neq\emptyset$. Therefore, $p\in \overline{S}$. 
	
Now we start the construction of $(B_i)_i$, $(S_i)_i$ and $(n_i)_i$.   We have already chosen $n_1$. Since  the topology on $X_{n_1}$ is analytic, by Lemma \ref{jalaili2}, there exists  $D\su \N$ such that  
\begin{equation*}
\bigcup_{k\in D}G_k \not\in \ideal_{p_{n_1}}\ \text{ and }\ \bigcup_{k\not\in D} G_k \not\in \ideal_{p_{n_1}}.
\end{equation*}
Since $A=(\bigcup_{k\in D }F_k)\cup (\bigcup_{k\not\in D}F_k)$, we have that $\bigcup_{k\in D}F_k\not\in \ideal_p$ or $\bigcup_{k\not\in D}F_k\not\in \ideal_p$. Without loss of generality, we suppose that $\bigcup_{k\not\in D}F_k\not\in \ideal_p$ and we let $B_1=\N\setminus D$. 
Since $\bigcup_{k\in D}G_k \not\in \ideal_{p_{n_1}}$ and $X_{n_1}$ is a \qq space, there exists $S_1\not\in\ideal_{p_{n_1}}$ such that $S_1\subset \bigcup_{k\in D} G_k$ and $|S_1\cap G_k|\leq 1$, for all $k\in D$.

We will present one more step in the construction and the pattern will be clear. 
Let $A_1=\bigcup_{k\not\in B_1}F_k$.  By construction,  $p\in\overline{A_1}$.  As  in the first step, we can assume that there is no 
sequence $(z_n)_{n\in\N}$ in $ A_1\setminus\{p\}$ such that $\lim_{n\to\infty}z_n=p$. Thus we can apply Lemma \ref{lema000} to get $n_2> n_1$ such that $p_{n_2}\in \overline{\pi_{n_2}(\bigcup_{i\in\mathcal{B}_1}F_i)\setminus \{p_{n_2}\}}$.  We repeat the same argument used in the first step but now to the set $A_1$ and the partition $(F_k)_{k\not\in B_1}$ and  find $B_2\subset \N\setminus B_1$ and a partial selector $S_2\subseteq X_{n_2}$ for the partition  $(G_k)_{k\not\in B_1}$.
\end{proof}

\section{Countable fan-tightness}

In this section we show that  countable fan-tightness is preserved under inverse limits. We star with an auxiliary result. 

\begin{lemma}\label{lema002}
Let $f\in Y^X$ and let $(A_n)_{n\in\N}$ be a decreasing sequence of subsets of $X$. Suppose   $B\subseteq f(A_0)$ and  $B\subseteq^*f(A_n)$, for each $n\in\N$. If  $B\cap (\cap_{n\in\N}f(A_n))$ is (at most) countable, then there exists $D\subseteq A_0$ such that $D\subseteq^*A_n$, for each $n\in\N$,  and $f(D)=B$.
\end{lemma}

\begin{proof}
Let  $B\cap (\cap_{n\in\N}f(A_n))=\{b_n : n\in\N\}$. Since $b_n\in f(A_n)$, for each $n\in\N$, let $a_n\in A_n$ be such that $f(a_n)=b_n$. Let $F_n=B\cap (f(A_n)\setminus f(A_{n+1}))$, for each $n\in\N$. Since $B\subseteq^*f(A_{n+1})$, then $F_n$ is finite, for each $n\in\N$. Let $L_n\subseteq A_n\setminus A_{n+1}$ be a finite set such that $f(L_n)=F_n$, for each $n\in\N$.

Let $$
D=\left(\cup_{n\in\N}L_n\right)\cup\{a_n : n\in\N\}.
$$
It is clear that  $D\subseteq A_0$. We check that $D\subseteq^*A_n$, for each $n\in\N$.  In fact,  let $i_0\in\N$. Since $(A_n)_{n\in\N}$ is a decreasing sequence, we have that $L_i\subseteq A_{i_0}$, for each $i\geq i_0$. Hence, $\cup_{i\geq i_0}L_i\subseteq A_{i_0}$. Furthermore, $\{a_n :n\in\N\}\subseteq^*A_{i_0}$ and $L_i$ is finite, for each $i\in\{0,...,i_0-1\}$. Therefore, $D\subseteq^*A_{i_0}$, and $D\subseteq^*A_{j}$ for each $j\in\N$.

Finally, we verify that $f(D)=B$. Since $f(L_n)=F_n\subseteq B$ and $f(a_n)=b_n\in B$, for each $n\in\N$, we have that $f(D)\subset B$. Conversely,  let $y\in B$. If $y\in F_n$ (i.e., $y\in f(A_n)\setminus f(A_{n+1})$), for some $n$, then, since  $f(L_n)=F_n$ and $L_n\subseteq D$,  $y\in f(D)$. Otherwise,  $y\in \cap_{n\in\N}f(A_n)$, then $y=b_k$, for some $k\in\N$. Thus, $y=f(a_k)$ and $y\in f(D)$. Therefore, $B\subseteq f(D)$ and $f(D)=B$.
\end{proof}

\begin{theorem}
Let $\{X_n;f_n^{n+1}\}_{n\in\N}$ be an inverse sequence of countable spaces and let $X_{\infty}$ be the inverse limit of $\{X_n;f_n^{n+1}\}_{n\in\N}$. If $X_n$ has countable fan-tightness, for each $n\in\N$, then $X_{\infty}$ also has it.
\end{theorem}

\begin{proof}
Let $z=(z_n)_{n\in\N}\in X_{\infty}$, we will show that $\ideal_z$ is \pp. Let $(A_n)_{n\in\N}$  be a sequence of subsets of $X$ such that $A_{n+1}\subseteq A_n$ and $z\in \overline{A_n}$, for each $n\in\N$. We have to show there is  $A\not\in\ideal_z$ such that $A\subseteq^*A_n$, for each $n\in\N$. 
We consider two cases.

\medskip

{\em Case 1.} For each $k\in\N$, there exist $n_k, m_k\geq k$, such that $z_{n_k}\notin \overline{\pi_{n_k}(A_{m_k})\setminus\{z_{n_k}\}}$. Let $U_{n_k}$ be an open set such that $z_{n_k}\in U_{n_k}$ and $U_{n_k}\cap (\pi_{n_k}(A_{m_k})\setminus\{z_{n_k}\})=\emptyset$, for each $k\in\N$. Since $A_{m_k}\not\in\ideal_{z}$, $\pi_{n_k}^{-1}(U_{n_k})\cap (A_{m_k}\setminus\{z\})\neq\emptyset$, for each $k\in\N$.  Let $a_k\in \pi_{n_k}^{-1}(U_{n_k})\cap (A_{m_k}\setminus\{z\})$, for each $k\in\N$ and  $A=\{a_k : k\in\N\}$.  Since $m_k\geq k$, for each $k\in\N$, and $(A_n)_{n\in\N}$ is a decreasing sequence, we have that $A_{m_l}\subseteq A_n$, for each $l\geq n$. Thus, $A\subseteq^*A_n$ for each $n\in\N$. 

We  will show that $z\in \overline{A}$. Let $W$ be an open set such that $z\in W$. Then, there exist $l\in\N$ and an open set $V_l\subseteq X_l$, such that $z_l\in V_l$ and $\pi_l^{-1}(V_l)\subseteq W.$ We prove that $a_l\in W$. Let $n_l, m_l\geq l$. By construction,  $z_{n_l}\notin \overline{\pi_{n_l}(A_{m_l})\setminus\{z_{n_l}\}}$. Since $f_l^{n_l}\colon X_{n_l}\to X_l$ is a continuous map, there exists an open set $V_{n_l}$ such that $z_{n_l}\in V_{n_l}$ and $f_l^{n_l}(V_{n_l})\subset V_l$; i.e., $\pi_{n_l}^{-1}(V_{n_l})\subseteq \pi_l^{-1}(V_l)\subseteq W$.

We claim that 
\begin{equation}\label{eq012}
\pi_{n_l}^{-1}(U_{n_l})\cap A_{m_l}\subseteq \pi_{n_l}^{-1}(U_{n_l}\cap V_{n_l})\cap A_{m_l}.
\end{equation}
In fact, let $y\in \pi_{n_l}^{-1}(U_{n_l})\cap A_{m_l}$. Hence, $\pi_{n_l}(y)\in U_{n_l}\cap \pi_{n_l}(A_{m_l})$. Since $U_{n_l}\cap (\pi_{n_l}(A_{m_l})\setminus\{z_{n_l}\})=\emptyset$, then  $\pi_{n_l}(y)=z_{n_l}$. Thus, $y\in  \pi_{n_l}^{-1}(U_{n_l}\cap V_l)\cap A_{m_l}$ and  (\ref{eq012}) holds.

Therefore,  $\pi_{n_l}^{-1}(U_{n_l})\cap A_{m_l}\subseteq \pi_{n_l}^{-1}(V_{n_l})\subseteq \pi_l^{-1}(V_l)\subseteq W$, by (\ref{eq012}). Hence, $a_{l}\in W$, $W\cap (A\setminus\{z\})\neq\emptyset$, and $z\in \overline{A}$.

\medskip

{\em Case 2.} There exists $k_0\in\N$ such that $\pi_n(A_m)\not\in\ideal_{z_n}$, for each $n,m\geq k_0$. 
Let ${P}_{k_0}, {P}_{k_0+1},...$ be a pairwise disjoint family of infinite subsets of $\N$ such that $\N=\bigcup_{i\geq k_0} {P}_i$. Let $j\geq k_0$.  Since $(\pi_j(A_m))_{m\in{P}_{j}}$ is a decreasing sequence  of sets not in  $\ideal_{z_j}$ and  $\ideal_{z_j}$ is \pp, there exists $B_j\not\in\ideal_{z_j}$ such that $B_j\subseteq^*\pi_j(A_n)$, for each $n\in {P}_{j}$. Let $n_{j}=\min {P}_{j}$, since $B_j\setminus \pi_j(A_{n_{j}})$ is finite, we may suppose that $B_j\subseteq \pi_j(A_{n_{j}})$. By Lemma \ref{lema002}, there exists $\BB_j\subseteq A_{n_{j}}$ such that $\BB_j\subseteq^*A_n$, for all  $n\in {P}_{j}$, and $\pi_j(\BB_j)=B_j$. Thus,  we have constructed a sequence  $\BB_j\subseteq A_{n_{j}}$, for $j\geq k_0$, where $n_{j}=\min {P}_{j}$. Let $\BB=\bigcup_{j\geq k_0}\BB_j$. We will show that $\BB\subseteq^* A_n$, for all  $n\in\N$ and $\BB\not\in \ideal_z$.

(a) Let $m_0\in\N$. Since $({P}_j)_{j\geq k_0}$  is a partition of $\N$, there is $l\geq k_0$ such that $\{0,...,m_0\}\subseteq \cup_{j=k_0}^l {P}_j$. For each $i\in\{k_0,...,l\}$, let $r_i\in {P}_i$ such that $r_i>m_0$. We  will show that 
$$
\BB\setminus A_{m_0}=\bigcup_{j\geq k_0} \left(\BB_j\setminus A_{m_0}\right)\subseteq \bigcup_{i=k_0}^{l}\left(\BB_{i}\setminus A_{r_i}\right).
$$
Let $x\in \BB\setminus A_{m_0}$. Hence, $x\in \BB_j$ for some $j\geq k_0$. If $j> l$, then $n_{j}=\min {P}_{j}>m_0$, hence $\BB_j\subseteq A_{n_{j_0}}\subseteq A_{m_0}$ and $\BB_j\setminus A_{m_0}=\emptyset$. Thus, $\bigcup_{j\geq k_0} (\BB_j\setminus A_{m_0})=\bigcup_{i=k_0}^{l} (\BB_i\setminus A_{m_0})$. Since $A_{r_j}\subseteq A_{m_0}$, for each $j\in\{k_0,...,l\}$, then  $\bigcup_{i=k_0}^{l} (\BB_i\setminus A_{m_0})\subseteq \bigcup_{i=k_0}^{l} (\BB_i\setminus A_{r_i})$. Therefore, $\BB\setminus A_{m_0}\subseteq \bigcup_{i=k_0}^{l}(\BB_i\setminus A_{r_i}).$  

Since $\BB_i\setminus A_{r_i}$ is finite, for all $i\in\{k_0,...,l\}$, we have that  $\BB\subseteq^*A_{m_0}$. Therefore, we have shown that $\BB\subseteq^* A_n$, for all  $n\in\N$.

(b)  Let $W$ be an open set with  $z\in W$. Thus, there exist $j\geq k_0$ and an open set $V_j\subseteq X_j$ such that $z_j\in V_j$ and $\pi_j^{-1}(V_j)\subseteq W$. Since $B_j\in\mathcal{C}_{z_j}$, $B_j\cap V_j\neq\emptyset$. As  $\pi_j(\BB_j)=B_j$, then  $\BB_j\cap \pi_j^{-1}(V_j)\neq\emptyset$ and $\BB\cap W\neq\emptyset$. Therefore, $\BB\not\in \ideal_z$.
\end{proof}

\section{Discrete generation and selective separability}

The product of two discretely  generated spaces is not necessarily discretely generated.   In fact,   Murtinova  \cite{murtinova2006} constructed  under CH two countable Fr\'echet spaces whose product is not discretely generated.  Our next result shows that inverse limit preserves discrete generation. 

\begin{theorem}\label{teoremadg}
	Let $\{X_n;f_n^{n+1}\}_{n\in\N}$ be an inverse sequence such that $X_n$ is a regular space, for each $n\in\N$. If $X_n$ is discretely generated, for each $n\in\N$, then the inverse limit $X_{\infty}=\underleftarrow{\lim}\{X_n;f_n^{n+1}\}_{n\in\N}$ is also discretely generated. 
\end{theorem}

\begin{proof}
	Let $p=(p_n)_{n\in\N}\in X_{\infty}$ and let $A\subseteq X_{\infty}$ such that $p\in \overline{A\setminus\{p\}}$. We will show that there is a discrete  $D \subseteq A\setminus \{p\}$ such that $p\in \overline{D}$.
	
	\medskip
	
	\noindent {\em Case 1:} Suppose there  exists an increasing sequence $(n_k)_{k\in\N}$ in $\N$ and $V_k\subseteq X_\infty$  open with $p\in V_k$, for $k\in \N$ such that  $p_{n_k}\notin \overline{\pi_{n_k}(A\cap V_k)\setminus \{p_{n_k}\}}$ for  all $k\in\N$.
	
	In this case, by Lemma \ref{lema000}, there is a sequence in $A\setminus\{p\}$ converging to $p$ which provides the required discrete set. 
	
	\medskip

	\noindent {\em Case 2:} There  exists $n_0$ such that  $p_{n}\in \overline{\pi_{n}(A\cap V)\setminus \{p_{n}\}}$, for each $n\geq n_0$ and each $V\subseteq X_\infty$ with $p\in V$. 
	
	\medskip
	
	\noindent {\em Case 2a:} There are  $n\geq n_0$, $E\subseteq X_n$ discrete and $D\subseteq A\setminus \{p\}$ such that $\pi_n(D)=E$ and $|\pi_n^{-1}(x)\cap D|=1$ for all $x\in E$ and $p\in \overline{D}$.

	We claim that  $D$ is discrete and thus the conclusion holds. In fact, let $z\in D$. Since $E$ is discrete and $\pi_n(z)\in E$, there exists an open set $U\subseteq X_n$ such that $U\cap E=\{\pi_n(z)\}$. Since $|D\cap \pi_n^{-1}(x)|=1$, for each $x\in E$, we have that $\pi_n^{-1}(U)\cap D=\{z\}$. Therefore, $D$ is a discrete subset of $A$.
	
	\medskip
	
	\noindent {\em Case 2b:} For all $n\geq n_0$, all $E\subseteq X_n$ discrete and all $D\subseteq A\setminus \{p\}$ such that $\pi_n(D)=E$ and $|\pi_n^{-1}(x)\cap D|=1$ for all $x\in E$, we have that  $p\notin \overline{D}$.
	
	\medskip
	
	We will construct sequences of sets $(V_k)_{n\geq n_0}$, $(D_n)_{n\geq n_0}$ and $(E_n)_{n\geq n_0}$ such that 
	\begin{itemize}
		\item[(i)]  $V_n\subset X_\infty$ is  open, $p\in V_n$ and $\overline{V_{n+1}}\subseteq V_n$,

		\item[(ii)] $E_n\subseteq X_n\setminus\{p_n\} $ is discrete, $p_n\in \overline{E_n}$, $D_n\subseteq A\setminus\{p¬?\} $ is discrete, $\pi_n(D_n)=E_n$,  $|\pi_n^{-1}(x)\cap D_n|=1$ for all $x\in E_n$,      $D_n\subseteq V_{n}$ and $D_n\cap \overline{V_{n+1}}=\emptyset$,  
	\end{itemize}
	
	We start the construction.  Let $V_{n_0}=X$. Since $p_{n_0}\in \overline{\pi_{n_0}A\setminus \{p_{n_0}\}}$ and $X_{n_0}$ is discretely generated, there is $E_{n_0}\subseteq X_{n_0}\setminus\{p_{n_0}\} $ is discrete with $p_{n_0}\in \overline{E_{n_0}}$. Let $D_{n_0}\subseteq A\setminus\{p¬?\} $ such that  $|\pi_{n_0}^{-1}(x)\cap D_{n_0}|=1$ for all $x\in E_{n_0}$. By the argument used in Case 2a, we have that $D_{n_0}$ is discrete. By the hypothesis of case 2b, we know that $p\notin \overline{D_{n_0}}$. Let $V_{n_0+1}$ be an open set containing $p$ such that $\overline{V_{n_0+1}}\subseteq V_{n_0}$ and $\overline{V_{n_0+1}}\cap D_{n_0}=\emptyset$. Now we repeat the process inside $A\cap V_{n_0+1}$ and find $D_{n_0+1}$ and $E_{n_0+1}$ as required.  The rest of the construction is analogous to the first step.

	Let $D=\bigcup_{n\geq n_0}D_{n}$. Since $D\cap (X\setminus \overline{V_{n}})= D_{n_0}\cup \cdots\cup D_{n-1}$, then $D$  is discrete. Finally, we prove that $p\in\overline{D}$. Let $m\in \N$ and  $U_{m}$ be an open subset of $X_{m}$ such that $p\in\pi_m^{-1}(U_m)$. Let $n>\max\{n_0,m\}$. Since $f_m^{n}\colon X_{n}\to X_m$ is continuous, there exists an open set $U_{n}$ such that $p_{n}\in U_{n}$ and $f_m^{n}(U_{n})\subseteq U_m$. Thus, $\pi_{n}^{-1}(U_{n})\subseteq \pi_m^{-1}(U_m)$. Since $p_{n}\in\overline{E_{n}}$, then  $U_{n}\cap E_{n}\neq\emptyset$. Hence, $\pi_{n}^{-1}(U_{n})\cap D_{n}\neq\emptyset$. Therefore, $D\cap \pi_{n}^{-1}(U_{n})\neq\emptyset$ and $p\in\overline{D}$.
\end{proof}

We recall that an heredirary class of spaces is closed under countable products iff it is closed under finite products and inverse limits.  Since DG is a hereditary property, then we immediately get the following.

\begin{corollary} (Alas-Wilson \cite[Theorem 2.5]{AlasWilson2013})
	Let $X_n$ be a discretely generated regular space for each $n\in\N$. Suppose that for each $m$, $\prod_{i=1}^m X_i$ is discretely generated. Then $\prod_{i=1}^\infty X_i$ is discretely generated.
\end{corollary}

Now we show that $SS$ is preserved under inverse limits. We denote by $\mathcal{DS}(X)$ the collection  of all dense subsets of $X$. We start with an auxiliary result.

\begin{lemma}\label{lema00}
     Let $X_{\infty}=\underleftarrow{\lim}\{X_n;f_n^{n+1}\}$. Then , $E\subseteq X_{\infty}$ is dense, if and only if $\pi_n(E)\subseteq X_n$ is dense, for each $n\in\N$.
\end{lemma}

\begin{proof}
	If $U\cap \pi_n(E)=\emptyset$, for some $n\in\N$, then $\pi_n^{-1}(U)\cap E=\emptyset$. Conversely, let $V$ be an open set of $X_{\infty}$. Then, there are $n_0\in\N$ and $U\subseteq X_{n_0}$ open, such that $\pi_{n_0}^{-1}(U)\subseteq V$. Since $\pi_{n_0}(E)\in \mathcal{DS}(X_{n_0})$, $\pi_{n_0}(E)\cap U\neq\emptyset$. Therefore, $E\cap \pi_{n_0}^{-1}(U)\neq\emptyset$ and $E\in \mathcal{DS}(X_{\infty})$.
\end{proof}

\begin{theorem}\label{SS}
Let $X_{\infty}=\underleftarrow{\lim}\{X_n;f_n^{n+1}\}$, where $X_n$  is  $SS$. Then  $X_{\infty}$ is $SS$. Moreover, if each $X_n$ is hereditarely SS, then so is $X_\infty$. 
\end{theorem}

\begin{proof}
Let $(D_n)_{n\in\N}$ be a sequence in $\mathcal{DS}(X_{\infty})$. Let $\{\mathcal{L}_n : n\in\N\}$ be a partition of $\N$ such that $\mathcal{L}_n$ is infinite, for each $n\in\N$. Observe that $(\pi_i(D_n))_{n\in\mathcal{L}_i}$ is a sequence in $\mathcal{DS}(X_i)$, for each $i\in\N$. Fix $i\in\N$. Since $X_i$ is $SS$, there exists $K_n\subseteq \pi_i(D_n)$ finite, for each $n\in\mathcal{L}_i$, such that  $\bigcup_{n\in\mathcal{L}_i}K_n\in \mathcal{DS}(X_i)$. For each $n\in\mathcal{L}_i$, let $N_n\subseteq D_n$ finite, such that $\pi_i(N_n)=K_n$.  It is clear that  $\bigcup_{n\in\N}N_n$ is dense in $X_{\infty}$, by Lemma \ref{lema00}. 

Suppose now that each $X_n$ is hereditarely SS and let $Z\subseteq X_{\infty}$.  Then  $\overline{Z}=\underleftarrow{\lim}\{\overline{\pi_n(Z)};f_n^{n+1}\restriction {\overline{\pi_{n+1}(Z)}}\}$ (see  \cite[Proposition 2.5.6, p.100]{engelking}). Hence,  $\overline{Z}$ is $SS$. Since $Z$ is dense in $\overline{Z}$, $Z$ is $SS$. Therefore, $X_{\infty}$ is hereditarely $SS$.
\end{proof}


%

\section{Examples}

We  present examples of inverse limit of the form $\underleftarrow{\lim}\{Z,f\}$, where $Z$ is a countable space with only one accumulation point.  But first,  we collect some facts about inverse limits of countable spaces that have only one non isolated point. 

\subsection{Inverse sequences of filters}
Suppose $Z=\N\cup \{\infty\}$ is a space such that  $\infty $ is the only accumulation point.  Then  $\mathcal{F}_\infty=\{A\su \N:\; \infty \in int_Z(A\cup \{\infty\} )\}$ is  the neighborhood filter of $\infty$. 
Conversely, given a filter $\mathcal{F}$ over $\N$,  we define  a topology on $\N\cup \{\infty\}$  by declaring that  each $n\in \N$ is isolated and $\mathcal{F}$ is the neighborhood filter of $\infty$.  We denote this space by $Z(\mathcal{F})$.   This is done analogously  on any countable set $X$ instead of $\N$. 

Given two ideals $\mathcal{I}$ and $\mathcal{J}$ on countable sets $X$ and $Y$ respectively, we say that $\mathcal{I}$ is below $\mathcal{J}$ in the Kat\v{e}tov order, denoted by $\mathcal{I}\leq_K \mathcal{J}$, if there is $f:Y\rightarrow X$ such that $f^{-1}(E)\in \mathcal{J}$ for all $E\in \mathcal{I}$. For more information about this pre-order we refer the reader to \cite{GuzmanMeza2016, HMTU2013,  Hrusak2017}. Let $\mathcal{I}^*$ and $\mathcal{J}^*$  be the  corresponding dual filters. We will abuse the notation and consider $f:Z(\mathcal{J}^*)\rightarrow Z(\mathcal{I}^*)$  by letting $f(\infty)=\infty$, then  $f:Z(\mathcal{J}^*)\rightarrow Z(\mathcal{I}^*)$ is clearly continuous. Conversely, if there is  $f:Z(\mathcal{J}^*)\rightarrow Z(\mathcal{I}^*)$ continuous with $f(\infty)=\infty$  and $f(Y)\su X$, then  $\mathcal{I}\leq_K \mathcal{J}$.  

\begin{proposition}
	\label{closedopen}
	Given two ideals $\mathcal{I}$ and $\mathcal{J}$ on countable sets $X$ and $Y$ respectively. 
	Let $f:Y\rightarrow X$. 
	\begin{enumerate}
		\item If $f:Z(\mathcal{J}^*)\rightarrow Z(\mathcal{I}^*)$ is continuous, then $\mathcal{I}\leq_K \mathcal{J}$. 
		
		\item If $f$ witnesses  that $\mathcal{I}\leq_K \mathcal{J}$, then $f:Z(\mathcal{J}^*)\rightarrow Z(\mathcal{I}^*)$ is continuous.
		
		\item If  $f$ is onto and $f:Z(\mathcal{J}^*)\rightarrow Z(\mathcal{I}^*)$ is closed, then $\mathcal{J}\leq_K \mathcal{I}$ and $f$ is open. 
		
		\item  If  $f$ is injective  and $f:Z(\mathcal{J}^*)\rightarrow Z(\mathcal{I}^*)$ is open, then $\mathcal{J}\leq_K \mathcal{I}$ and $f$ is closed. 
		
	\end{enumerate}	
\end{proposition}

\proof (1) and (2) are straightforward. To see (3), let $g:X\rightarrow Y$ be such that $f\circ g=1_X$. Let $D\su Y$ in $\mathcal{J}$,  then $g^{-1}(D)\su f(D) $ for $D\su Y$.  Since $D$ is closed in  $Z(\mathcal{J}^*)$, then $f(D)$ is closed in $Z(\mathcal{I}^*)$. Hence $f(D)\in \mathcal{I}$ and thus $g^{-1}(D)\in \mathcal{I}$.  We have shown that  $\mathcal{J}\leq_K \mathcal{I}$. To see that $f$ is open, observe that  $X\setminus f(A)\su f(Y\setminus A)$ for all $A\su Y$.

For (4), let $g:X\rightarrow Y$ be such that $g\circ f =1_Y$. Let $D\su Y$, then $f(Y\setminus D)\su g^{-1}(Y\setminus D)=X\setminus g^{-1}(D)$. As $f$ is open, if  $D\in \mathcal{J}$, then  $f(Y\setminus D)\in \mathcal{I}^*$, thus $g^{-1}(D)\in \mathcal{I}$. We have shown that  $\mathcal{J}\leq_K \mathcal{I}$. To see that $f$ is closed, observe that  $f(Y\setminus A)\su X\setminus f(A)$ for all $A\su Y$.
\endproof

From the previous observation we get that any sequence of ideals $(\mathcal{I}_n)_n$ together  with functions $f_n$ witnessing that $\mathcal{I}_n\leq_K \mathcal{I}_{n+1}$ corresponds to an inverse sequence  $\{ Z(\mathcal{I}_n^*), f_n\}_n$.   There are strictly increasing transfinite sequences in the Kat\v{e}tov order of length $\omega_1$ (see \cite{GuzmanMeza2016}). Next proposition says that any countable increasing sequence has an upper bound which can be defined using an inverse limit.

\begin{proposition}
	Let $\{\mathcal{I}_n, f_n\}$ be a $\leq_K$-increasing  sequence of ideals  over $\N$. Let $X_\infty$ be the inverse limit of $\{ Z(\mathcal{I}_n^*), f_n\}_n$. Let $p\in X_\infty$ be the constant sequence $\infty$. There is a countable discrete set $D\su X_\infty\setminus \{p\}$ such that $p\in \overline{D}$ and 
	\[
	\mathcal{I}_n\leq_K \mathcal{J}_p
	\]
	for all $n$, where $\mathcal{J}_p$ is the dual ideal of the neighborhood filter of $p$ in the space $D\cup\{p\}$. 
	
\end{proposition}

\proof
For each $n\in \N$, let $x_n\in \pi^{-1}_1(n)$ and  $D=\{x_n:\; n\in \N\}\cup \{ p\}$. Clearly $D$ is discrete and  $p\in \overline{D}$. Let $f_n=\pi_n\restriction D: D\rightarrow \pi_n(D)$. Then $f_n$ witness that  $\ideal_n\restriction \pi_n(D)\leq_K \mathcal{J}_p$ for all $n$.  Clearly  $\ideal_n\leq_K \ideal_n\restriction \pi_n(D)$  and therefore
$\mathcal{I}_n\leq_K \mathcal{J}_p$ for all $n$. 
\endproof

Given a filter $\mathcal{F}$, the following result will be used to  construct a countable spaces $X$ without isolated points such that $Z(\mathcal{F})$ embeds into $X$. 

\begin{proposition}
	\label{embedding}
	Let  $\mathcal{F}$ be a filter on $\N$ and $Z=Z(\mathcal{F})$.  Let $f:Z\rightarrow Z$ be a continuous and closed surjection such that  $f(\infty)=\infty$. Let $X_\infty=\underleftarrow{\lim}\{Z,f\}$. For each $n\in \N$, pick $x_n\in \pi^{-1}_1(n)$ and let $Y=\{x_n:\; n\in \N\}\cup \{ p\}$, where $p\in X_\infty$ is the constant sequence $\infty$. Then the map $i: Z\rightarrow X_\infty$ given by $i(n)=x_n$ and $i(\infty)=p$ is an embedding.
\end{proposition}

\proof 
First, notice that each $x_n$ is isolated in $Y$, since  $\{x_n\}=\pi_1^{-1}(\{n\})\cap Y$.  Now we show that the collection of sets $\pi_1^{-1}(V)$, for $\infty\in V\su Z$ open, is a basis for $p$ in  $X_\infty$.   Since $f$ is continuous and onto, $f$ is open, by Proposition \ref{closedopen}. Thus, $f^{-1}(\infty)=\{\infty\}$.

Let $W\su Z$ open with $\infty \in W$ and $k\in \N, k>1$. Then $\pi_k^{-1}(W)$ is a basic open set containing $p$. We will show that there is $V\su Z$ open such that $\infty\in V$ and  $\pi_1^{-1}(V)\su \pi_k^{-1}(W)$. Let $U=(f^{k-1})^{-1}(f^{k-1}(W))$. Then $U$ is open and $\infty\in U$. Let $A=U\setminus  W$ notice that $\infty\not\in A$ and $A$ is closed (note that $A$ may be empty). Then  $V=f^{k-1}(W)\setminus f^{k-1}(A)$ is open and $\infty\in V$.  

We see that $\pi_1^{-1}(V)\su \pi_k^{-1}(W)$. Let $x\in X_\infty$ such that $\pi_1(x)\in V$ and suppose that $\pi_k(x)\notin W$. Since $f^{k-1}(\pi_k(x))=\pi_1(x)$ and $\pi_1(x)\in f^{k-1}(W)$, we have that $\pi_k(x)\in (f^{k-1})^{-1}(f^{k-1}(W))=U$. Thus, $\pi_k(x)\in U\setminus W=A$. Since $f^{k-1}(\pi_k(x))=\pi_1(x)$, $\pi_1(x)\in f^{k-1}[A]$. A contradiction.

To see that  $i$ is an embedding, we have to show that $A\in \mathcal{F}$  iff $\{x_n: \; n\in A\}\cup \{p\}$ is open in $Y$,  for any $A\su \N$.  In fact, $\{x_n: \; n\in A\}\cup \{p\}$ is open in $Y$ iff there is $B\in \mathcal{F} $ such that $\pi^{-1}_1(B \cup \{\infty\})\cap Y\su \{x_n: \; n\in A\}\cup \{p\}$ iff there is $B\in\mathcal{F}$ such that $B\su A$. 
\endproof

\begin{proposition}
	\label{enlacebiyectiva}
	Let $\mathcal{F}$ be a  filter on $\N$ and $Z=Z(\mathcal{F})$.  Let $f:Z\rightarrow Z$ be a continuous bijection. Let $p\in X_\infty$ be the constant sequence $\infty$.  Then $p$ is the only non isolated point of  $X_\infty=\underleftarrow{\lim}\{Z,f\}$.  The neighborhood filter, $\mathcal{F}_p$,  of $p$  is  generated by the sets $\pi^{-1}_1(f^{k}(W))$ for $W\in \mathcal{F}_\infty$ and $k\in \N$. Moreover,    
	\[
	\mathcal{F}_p=\pi^{-1}_1(\bigcup_k f^{k}(\mathcal{F})).
	\]
\end{proposition}	

\proof  Since $f$ is bijective, then  each $\pi_k$ is bijective, $X_\infty$ is countable and each $\pi^{-1}(n)$ is an isolated point. On the other hand, by Proposition \ref{baseIL}, $\mathcal{F}_p$ is generated by the collection of sets  $\pi_k^{-1}(W)$ for $W\su Z$ open with $\infty \in W$ and $k\in \N$. 
Since $f^{k-1}\circ \pi_k=\pi_1$, then $\pi_1^{-1}(f^{k-1}(A))=\pi_k^{-1}(A)$ for all $A\su Z$ and  we are done.  To verify  the last claim, we observe that, by the continuity of  $f$,  $\mathcal{F}\su f(\mathcal{F})$,  hence  $f^k(\mathcal{F}) \su f^{k+1}(\mathcal{F})$. Thus $ \bigcup_k \pi^{-1}_1(f^{k}(\mathcal{F}))$ is a filter and we are done. 
\endproof

\subsection{Some examples}

\begin{example}
\label{ejem-lastres}
There exists a   regular crowded countable space $X$ with analytic topology such that  it is \qq, discretely generated, $SS$ and  does not have countable fan-tightness.
\end{example}
The space $X$ will be a subspace of an  inverse limit of the form $\underleftarrow{\lim}\{Z,f\}$.  We first define the space $Z$. Consider the following ideal on $\N\times\N$: 
\[
\fin\times\fin=\{A\su\N\times\N:\; \{n\in \N:\; \{m\in\N:\; (n,m)\in A\}\not \in \fin\}\in \fin \}.
\]
Let $Z=Z(\mathcal{F})$ where $\mathcal{F}$ is the  dual filter of $\fin\times\fin$. It is clear that the topology of $Z$ is analytic.   To check that $Z$ is \qq,  it suffices to observe that it is (homeomorphic to) a subspace of Arens space $S_2$. On the other hand,  $\fin\times\fin$ is not \pp, in fact, consider the sets $A_n=\bigcup_{k\geq n} \{k\}\times \N\in \fin\times\fin$.  Hence $Z$  has no countable fan-tightness.

Let  $f\colon Z\to Z$  be any function such that:
\begin{enumerate}
	\item[(a)] $f(\infty)=\infty$;
	\item[(b)]\label{ii} $f^{-1}(\{n\}\times\N)=\{n\}\times\N$, for each $n$, and $f^{-1}((n,m))$ has two points for each $(n,m)\in \N^2$.
\end{enumerate}
Let $f_n^{n+1}=f$, for each $n\in\N$. It is not difficult to see that $f$ is a continuous, closed and open surjection. Let $X_\infty =\underleftarrow{\lim}\{Z,f\}$.  By Proposition \ref{inverseAnalytic}, Theorems \ref{teoremaq}, \ref{teoremadg}, and \ref{SS}, every countable subspace of  $X_\infty$ has a $F_\sigma$ basis,   is  \qq,  discretely generated and $SS$.

By Proposition \ref{embedding}, there is an embedding $i\colon Z\to X_\infty$ such that $|\pi_1^{-1}(n,m)\cap i(Z)|=1$, for each $(n,m)\in\N\times\N$. By (b), $X_\infty$ has no isolated points, so let $D_{n,m}\su \pi^{-1}_1(n,m)$ be a countable set without isolated points, such that $|D_{n,m}\cap i(Z)|=1$, for each $(n,m)\in \N\times\N$. Let $X=\bigcup_{(n,m)\in\N\times\N}D_{(n,m)}\cup\{p\}$. Since $i(Z)\subseteq X$ and $i(Z)$ has not countable fan tightness, then  $X$ neither has it.

\bigskip
%
%

\begin{example}
	\label{ejemplo2}
There exists a regular  crowded countable space $X$  with a $F_\sigma$ basis  such that it is  discretely generated, $SS$,  has countable fan-tightness and  is not \qq.
\end{example}

We first define an ideal on $\N$. Let $(L_n)_n$ be a partition of $\N$ such that $|L_n|=2^n$ for all $n\in \N$.  Consider the following ideal on $\N$: 
\[
\ideal=\{A\su\N:\; \exists k\in\N, \forall n\in \N\; |A\cap L_n|\leq k\, \}.
\]
Let $Z=Z(\mathcal{I}^*)$. It is clear that the topology of $Z$ is analytic.    It is obvious that  $Z$ is not \qq.  Notice that $\ideal$ is a $F_\sigma$ subset of $\cantor$, therefore $Z$ is \pp\ (this is well known, see for instance, \cite[Lemma 3.3]{HMTU2013}). Moreover,  by Proposition \ref{inverseAnalytic},  every countable subspace of $X_\infty$ has a $F_\sigma$ basis. 

Let  $f\colon Z\to Z$  be any surjective function such that:
\begin{enumerate}
	\item[(a)] $f(\infty)=\infty$;
	\item[(b)]\label{ii} $f^{-1}(L_{n})=L_{n+1}$, for each $n$, and $f^{-1}(x)$ has two points for each $x\in \N$.
\end{enumerate}
Let $f_n^{n+1}=f$, for each $n\in\N$. It is not difficult to see that $f$ is a continuous, closed and open map. Let $X_\infty =\underleftarrow{\lim}\{Z,f\}$.  It is obvious that $Z$ is hereditarely SS. By Proposition \ref{inverseAnalytic}, Theorems \ref{teoremadg} and \ref{SS}, every countable subspace of  $X_\infty$  has analytic topology, is discretely generated and $SS$.

By (b), $X_\infty$ has no isolated points. For each $n\in \N$, let $D_n\su \pi^{-1}_1(n)$ be  a countable set without isolated points.  Let $X=\bigcup_{n\in\N}D_n\cup\{p\}$. By   Proposition \ref{embedding}, there is an embedding from $Z$ into $X$, therefore $X$ is not \qq.

\begin{example}
\label{ejemplo3}
There exists a regular  crowded countable space $X$  such that it is  discretely generated, $SS$,  has not  countable fan-tightness and   is not \qq.
\end{example}

Let $(L_n)_n$ be a partition of $\N$ such that $|L_n|=2^n$ for all $n\in \N$ and  $\ideal$ be the ideal as defined in Example \ref{ejemplo2}. Consider the following ideal on $\N\times\N$: 
\[
\mathcal{J}=\{A\subseteq \N\times\N:\; \forall n\in\N, \{m\in\N:\; (n,m)\in A\}\in\ideal\}.
\]
It is easy to verify that $\mathcal{J}$ is neither $q^+$ nor $p^+$. 

Let $f:\N\rightarrow \N$ be an onto map  such that $f^{-1}(L_{n})=L_{n+1}$, for each $n$, and $f^{-1}(x)$ has two points for each $x\in \N$.
Let $Z=Z(\mathcal{J}^*)$. Consider $g:Z\rightarrow Z$ given by $g(n,m)=(n,f(m))$ and $g(\infty)=\infty$.  Let  $X_\infty =\underleftarrow{\lim}\{Z,g\}$. As in the previous example, let $D_{(n,m)}\su \pi^{-1}_1((n,m))$ be  a countable set without isolated points, for each $(n,m)\in \N\times\N$.  Let $X=\bigcup_{(n,m)\in\N\times\N}D_{(n,m)}\cup\{p\}$. By   Proposition \ref{embedding}, there is an embedding from $Z$ into $X$. Therefore $X$ is the required space.

\bigskip

Every countable sequential space is \qq\  (see \cite[Proposition 3.3]{Todoruzca2000}). Our last example  shows that sequentiality is not preserved by inverse limits, this was known (see \cite[3.3.E(b), p.156]{engelking}), but our example is different.

\begin{example}
\label{ejem-nosequential}
Let $Z$ be the countable sequential fan, i.e., $Z=Z(\mathcal{F})$ where $\mathcal{F}$ is the filter on $\N\times\N$ given by 
\[
A\in \mathcal{F}\;\;\Leftrightarrow\;\; \forall n\in\N , \{m\in \N: (n,m)\not \in A\}\in \fin.
\]
Let  $f:Z\rightarrow Z$ be given by $f(\infty)=\infty$ and 
\begin{equation*}
f(n,m)=\begin{cases}
(0,n) &\text{ if } m=0;\\
(n+1,m-1) &\text{ if } m\neq 0.
\end{cases}
\end{equation*}
Then $X_\infty =\underleftarrow{\lim}\{Z,f\}$ is not sequential.
\end{example}
Clearly $f$ is continuous and bijective. 
It is easy to verify that $A\in f^k(\mathcal{F})$ iff $\{m\in \N: (n,m)\not \in A\}$ is finite for all $n\geq k$. Therefore, by Proposition \ref{enlacebiyectiva},  $X_\infty=Z(\mathcal{F}_p)$ and  $\mathcal{F}_p$ is isomorphic to  $ \bigcup_k f^{k}(\mathcal{F})$ the later is clearly the dual filter of $\fin\times\fin$.  It is not difficult to verify that  $X_\infty$ is not sequential by embedding it into Arens space.

\end{document}